\documentclass[12pt]{article}
\usepackage{geometry}                
\usepackage{graphicx}
\usepackage{amssymb}
\usepackage{amsmath}
\usepackage{amsthm}
\usepackage{amsfonts}

\usepackage{amsmath,amsfonts, amssymb, amsthm, amscd,graphicx}
\usepackage{latexsym}
\usepackage{epsfig,epsf,psfrag}
\usepackage{nonfloat}
\newcommand{\R}{\ensuremath{\mathbb{R}}}

\newcommand\PP{{\mathbb P}}

 \newcommand{\eps}{\varepsilon}
 \newtheorem{theorem}{Theorem}[section]
 \newtheorem{lemma}[theorem]{Lemma}
 \newtheorem{propo}[theorem]{Proposition}
 \newtheorem{corollary}[theorem]{Corollary}

 \usepackage{eucal}

\title{Large deviations and path properties of the true self-repelling motion}
\author{Laure Dumaz\footnote {Ecole Normale Sup\'erieure, Universit\'e Paris-Sud and TU Budapest -- Support from the Balaton/PHC grant 19482NA is acknowledged.}}
\begin{document}

\maketitle

\begin {abstract}
 We derive some large deviation bounds for events related to the ``true self-repelling motion'', a one-dimensional self-interacting process introduced by 
T\'oth and Werner, that has very different path properties than usual diffusion processes. We then use these estimates to study certain of these path properties 
such as its law of iterated logarithms for both small and large times.
\end {abstract}

\section{Introduction}
In the present paper, we study some features of a self-interacting one-dimensional process called the true self-repelling motion, defined by T\'oth and Werner in 
\cite {TothWerner}. Let us first very briefly recall the intuitive definition of this process and describe the motivations that lead to our study. 

The true self-repelling motion is a continuous real-valued process $(X_t, t \ge 0)$ that is locally self-interacting with its past occupation-time. 
More precisely, for each positive time $t$, define its occupation-time measure $\mu_t$ that assigns to each interval $I \subset \R$, the time spent in it by $X$ before 
time $t$:
$$ \mu_t (I ) = \int_0^t 1_{X_s \in I}  \ ds .$$
It turns out that for this particular process $X$, almost surely for each $t$, the measure $\mu_t$ has a continuous density $L_t (x)$.
By analogy with semi-martingales, where such occupation-time densities also exist, the curve $x \mapsto L_t (x)$ is called the ``local-time'' profile of $X$ at time $t$. 
Heuristically, the dynamics of $X_t$ is such that the TSRM is locally pushed in the direction of the negative ``gradient'' of the local time at its current position.
Loosely formulated, one can write $dX_t = - \nabla_x L_t (X_t) dt$ (even if $(X_t, t \ge 0)$ is a random process).  
For more details and comments on this description, we refer to \cite {TothWerner}.
It turns out that this process is of a very different type than diffusions. For example (see again \cite {TothWerner}), its quadratic variation almost surely vanishes whereas its variation of power $3/2$ is positive and finite. Similarly, it does not have the Brownian scaling property, it has instead a $2/3$ scaling behavior i.e., for any positive $\lambda$, 
$ (X_{\lambda t}, t \ge 0) $ has the same law as $(\lambda^{2/3} X_t, t \ge 0)$. 

This same exponent $2/3$ appears in various other models that can be interpreted as continuous height-fluctuations of $1+1$-dimensional models 
in the Khardar-Parisi-Zhang universality class (such as the Tracy-Widom distribution for eigenvalues of large random matrices, the movement of the second-class particle in a TASEP etc.).
TSRM seems however at present to be one of the few such ``non-diffusive'' continuous processes that probabilists can define (see also 
\cite {ETW} for related questions).
All this gives us some motivation to study in more detail its behavior, in order to see what features it shares with the 
other previously-mentioned models, and also for its own independent interest. 

Let us now describe briefly the results of the present paper: 
Both for the process $(X_t, t \ge 0)$ itself as for the height process $(H_t, t \ge 0)$, we give upper and lower bounds for 
the probability that their value at a given time is very large. Combined with $0-1$-law arguments, this enables us 
to derive almost sure fluctuation results  (of the type of the law of the iterated logarithm) for these two processes. For instance, we shall see that 
$\limsup_{t \to \infty} X_t / (t^{2/3} (\log \log t)^{1/3})$ is almost surely equal to a finite  positive constant, and a similar result when $t \to 0$.

The construction of the process $X_t$ is based on a family of coalescing one-dimensional Brownian motions starting from all points in the plane. 
Such families had been constructed by Arratia in \cite {Arratia}, and further studied in \cite {TothWerner, STW, NewmanBW, NewmanBW2} 
and are called ``Brownian web'' in the latter papers. 
As a consequence, the estimates on the TSRM will follow from results concerning this Brownian web. In Section \ref {S.2}, 
we will recall some aspects of the construction of TSRM and some features of the Brownian web. In Section \ref {S.3}, 
we will focus on the large deviation estimates concerning $X_1$, we then derive the LIL for $X$ in Section \ref{S.4}, and we finally focus on 
the fluctuations of the height-process in the final Section \ref {S.5}.

\medbreak
\noindent \textbf{Acknowledgement:} I am grateful to my supervisors B\'alint T\'oth and Wendelin Werner for their guidance throughout this work. Special thanks go to Wendelin Werner for his careful reading of successive versions of this paper, and to the referees for their insightful comments. 

\section{Preliminaries and notations}\label{sec:preliminariesandnotations} \label {S.2}

In this section, we put down some notation, and collect some elementary estimates that will be useful later on.

\subsection{Versions of the Brownian web}

  The true self-repelling motion (TSRM) is a deterministic function of a certain family of coalescing one-dimensional Brownian motions.
 There are two natural variants of TSRM, that respectively correspond to such Brownian families in the entire plane (this is the ``stationary'' TSRM, this version has stationary increments) or in the upper half-plane (this is the TSRM with ``zero-initial conditions''). Other initial conditions are also possible, see Section 4 of \cite{STW} for examples.
   
   Let us briefly first recall the construction in the \textbf{stationary} case which will be the main focus of this paper.
 To start with, choose any deterministic countable dense family $Q$ of points $(\tilde{x},\tilde{h})$ in the plane, say $Q = \mathbb{Q}^2$. 
It is then possible to define 
the joint law of a family  $(\Lambda_{\tilde{x},\tilde h} ( \cdot), (\tilde{x},\tilde{h}) \in Q)$ in such a way that, for each $(\tilde{x},\tilde{h}) \in Q$,
 $\Lambda_{\tilde{x},\tilde{h}}$ is a function from $[ \tilde x, \infty)$ into $\R$, that is distributed like a Brownian motion started from height 
$\tilde{h}$ at time $\tilde{x}$. Furthermore (see e.g. \cite {TothWerner} for details), different curves are ``independent until their 
first meeting time'' and they coalesce after this meeting time (and follow the same Brownian evolution). 
Recall that $Q$ is dense in the plane, so that the picture of all these lines is dense in the plane. The coalescent structure nevertheless defines 
a tree-like structure rooted ``at $x=+ \infty$''. 
This family of curves $\Lambda$ is often referred to as the ``forward lines''. 

If we are given a countable dense family $\tilde Q$ in the plane, then one can almost surely define the family of ``backward'' lines 
$(\Lambda_{\tilde{x},\tilde{h}} (\cdot), (\tilde x, \tilde h) \in  \tilde Q)$ such that each $\Lambda_{\tilde x, \tilde h}$ is now a 
function defined on $(-\infty, \tilde x]$ in such a way that the backward lines can be viewed as the ``dual tree'' of the previous dense tree 
(it is therefore a deterministic function of all forward lines). It is proved in \cite{TothWerner} that this family of backward lines has the same 
law as the reversed image (changing $x$ into $-x$) of the law of the forward lines (choosing $\tilde Q$ to be the symmetric image of $Q$).

There is an alternative construction where one does not have to first define the whole dense family of forward lines to construct the backward ones: 
instead, one can construct the forward and the backward paths one by one for each $(\tilde{x},\tilde{h}) \in Q$ inductively, 
applying a reflection/coalescence rule explained in the section 3.1.4 of \cite{STW}. 
Roughly, the rule is that when two curves meet, there is coalescence if they are of the same type (both backward or both forward), and otherwise, the two curves are ``reflected on each other''. Note that the proofs in \cite{STW} use the discrete model (with reflecting/coalescing random walks) and an invariance principle.
   
Both constructions define (for each $Q$) a family of curves $\Lambda_{\tilde x, \tilde h} (\cdot)$ (from $\R$ to $\R$) indexed by $(\tilde x, \tilde h) \in Q$,  such that 
for each $(\tilde x, \tilde h)$ in $Q$, $\Lambda_{\tilde x, \tilde h}(\tilde x) = \tilde h$ almost surely.
It is then natural to wonder whether there exists certain ``versions'' of the process $(\Lambda_{x, h},\; (x, h) \in \R^2)$, 
defined simultaneously for all points $(x, h)$ in the plane, with some additional regularity properties. 
It turns out that the situation is reminiscent of that of 
real-valued L\'evy processes, where one can choose a right-continuous or a left-continuous version, except that time is here replaced by the $h$-variable.   

 In \cite{TothWerner}, the authors choose to define the forward line starting at $(x,h) \in \R^2$ denoted $\Lambda_{x,h}(y), y \ge x$ by taking the supremum of all $\Lambda_{\tilde{x}, \tilde{h}}(y)$ over the countable family of lines
   $$\{(\tilde{x},\tilde{h}) \in Q \;:\; \tilde{x} < x, \;\Lambda_{\tilde{x},\tilde{h}}(x) < h\}$$ that is to say over the lines in the countable family that are starting before $x$ and passing below $h$ at time $x$. Their Theorem 2.1 states that this family $\Lambda$ then verifies:
   \begin{itemize}
    \item for any finite set $(x_1,h_1), \cdots, (x_n,h_n) \in \R^2$, a.s. $(\Lambda_{x_i,h_i},\;i \in \{1,\cdots,n\})$ is distributed as independent coalescing Brownian motions,
    \item a.s., for all $(x,h) \in \R^2$, $\Lambda_{x,h}(x) = h$,
    \item a.s., for all $(x_1,h_1), (x_2,h_2) \in \R^2$,  $\Lambda_{x_1,h_1}$ and  $\Lambda_{x_2,h_2}$ do not cross each other,
    \item a.s., for all $x < y$, the mapping $h \mapsto \Lambda_{x,h}(y)$ is left-continuous,
   \end{itemize}
and that those four properties characterize its distribution. Note that the first one tells us that the choice of $Q$ does not change the distribution of $\Lambda$. The last ``left-continuity'' means that for those $(x,h)$ where there might be some choice, one chooses the lowest one. 
   Throughout our paper, the notation $(\Lambda_{x,h})_{(x,h) \in \R^2}$ corresponds to this version of the coalescing family.

 Clearly, there is another natural choice, that one can obtain by considering the symmetric picture (upwards down) i.e. to define 
$$\Lambda^+_{x,h} = \inf\{\Lambda_{\tilde{x},\tilde{h}}(y), \;\tilde{x} < x, \;\Lambda_{\tilde{x},\tilde{h}}(x) > h, \;(\tilde{x},\tilde{h}) \in Q\}.$$ 
This family $\Lambda^+$ verifies the same properties as $\Lambda$, except that left-continuity with respect to $h$ is replaced by right-continuity.
   
   Another option proposed by Fontes, Isopi, Newman, and Ravishankar in \cite{NewmanBW} is to define a metric on a natural space on which the 
coalescing family lives and to consider the closure of $(\Lambda_{x,h}(y), y \ge x\; ;\; (x,h) \in Q)$ in this topological space. Note that you can 
now have more than one curve starting from certain (exceptional) points. In fact, the curves of the families $\Lambda$ and $\Lambda^+$ correspond 
to the two extremal choices for the curves of their family. This construction is useful in order to state the convergence of the discrete model 
with coalescing random walks towards the coalescing Brownian motions. The family is called in their paper Brownian Web (Double Brownian Web if 
you add the backward lines). By a slight abuse of terminology, we will just call our family $(\Lambda_{x,h}(y), \;y \in \R \;;\; (x,h) \in \R^2)$ ``Brownian Web'' (BW).

\subsection {TSRM and the Brownian web}\label{subsec:TSRMandBW}
   
   The intuitive link between the TSRM and the BW goes as follows: Let us consider the process $(X_t,H_t)$ started at $(0,0)$ which traces 
the contour of the ``forward tree'' moving upwards, that is to say above $\Lambda_{0,0}$ and towards $+\infty$. It is in fact the same contour as that of the ``backward tree''. This process visits all 
the points above the curve $\Lambda_{0,0}$ (it is plane-filling). The time-parametrization will be chosen in such a way that the area swept by $(X,H)$ during the interval $[0,t]$ is exactly $t$ and its first coordinate $X$ will be the ``true'' self-repelling motion. 

In order to be more precise, we need some additional notations. For each $(x,h) \in \R^2$, let $S_{x,h}$ denotes the (algebraic) area 
between $\Lambda_{x,h}$ and $\Lambda_{0,0}$:
   $$S_{x,h} := \int_{-\infty}^{+\infty} (\Lambda_{x,h}(y) - \Lambda_{0,0}(y)) \, dy.$$
   Almost surely, for every $(x,h)$ above the initial curve $\Lambda_{0,0}$, the process $(X,H)$ is equal to $(x,h)$ at the random time $S_{x,h}$ 
and has visited all the points between 
$\Lambda_{x,h}$ and $\Lambda_{0,0}$.  T\'oth and Werner proved that this indeed defines a continuous process $(X_t, H_t)_{t \ge 0}$ (see Lemma 3.4 of 
\cite{TothWerner}). Thanks to the Brownian structure of the tree and the correspondence between area in the tree and time for the process, one can then easily 
deduce basic properties for $(X,H)$ such as the recurrence of $X$ in $\R$, or the scaling property: for every $a > 0$, $(X_{a t},H_{a t})_{t \ge 0}$ and 
   $(a^{2/3} X_{t}, a^{1/3} H_{t})_{t \ge 0}$ are identical in law (Proposition 3.5 of \cite{TothWerner}).
   
   \medskip
   
   Another important observation is that together with the initial profile $\Lambda_{0,0}$, the first coordinate $X$ contains enough information in order to recover
both the process $H$ and the upper part of the BW $(\Lambda_{x,h},\; x \in \R, \,h \ge \Lambda_{0,0}(x))$.
 Indeed, as we already mentioned in the introduction, the occupation-time measure of $X$ turns out (for each time $t$) to have a continuous density 
with respect to Lebesgue measure, denoted by $L_t(\cdot)$. Moreover, the definition of $(X,H)$ readily shows that when $t = S_{x,h}$, then 
$$ \Lambda_{0, 0} ( \cdot ) + L_t ( \cdot ) = \Lambda_{x,h} ( \cdot )$$
i.e. that the random area $S_{x,h}$ corresponds to the first time $t$ at which the local time at $x$, $L_t(x)$, reaches the level 
$h - \Lambda_{0,0} (x)$, and the curve of the BW from $\Lambda_{x,h} - \Lambda_{0,0}$ is the local time curve at $S_{x,h}$. It is a stronger analog to Ray-Knight Theorems for Brownian motion.

\bigskip
   
  For each fixed (deterministic) $x \in \R$, we will denote by $\sigma_x$ the first hitting time of $x$ by the TSRM $X$. 
It is easy to see that $a.s$, this time equals the infimum of the set of times at which  $L_{\cdot}(x)$, is positive. 
That is to say, for every given $x \neq 0$, $\sigma_x$ is almost surely equal to the infimum of $S_{x,h}$ over all $h > \Lambda_{0,0} (x)$ (note that this is not true for all $x$ simultaneously because of the existence of ``fast points'' or of local maxima).

In the sequel, we shall simply 
denote by $\Gamma_x (\cdot)$, the profile at this time $\sigma_x$: $$\Gamma_x (\cdot) := L_{\sigma_x}(\cdot) + \Lambda_{0,0} (\cdot).$$ Remark that almost surely for every $x \in \R$ this curve is equal to 
$\Lambda^+_{x, \Lambda_{0,0}(x)}(\cdot)$, coming from the right-continuous version of the BW (this is contained in Theorem 4.3 (ii) in \cite{TothWerner}). Note also that with this definition $\Gamma_0$ is just the same as the initial profile $\Lambda_{0,0}$. 

The following Lemma describes the joint law of $\Gamma_0$ and $\Gamma_x$. In fact, we will use a slightly stronger version and describe the law of $\Gamma_Y$, when $Y$ is a  for some $\Gamma_0$-measurable random variable $Y$:
\begin {lemma}\label{lem:BWalea}
 Let $Y$ denote a $\Gamma_0$-measurable random variable. Then, conditionally on $\Gamma_0$, the distribution of $\Gamma_Y$ is that of a 
coalescing-reflecting Brownian motion started from $(Y, \Gamma_0 (Y))$, that is reflected on $\Gamma_0$ in the interval between $0$ and $Y$ 
and coalescing with it outside of this interval.
\end {lemma}

As the ``starting point'' $(Y,\Gamma_0(Y))$ of $\Gamma_x$ is random, this fact is not totally straightforward. Our proof uses features of the BW established in \cite{TothWerner}.

\begin{proof}
We already know that for a \textbf{fixed} point $(x,h)$ in the plane and conditionally on $\Gamma_0$, $\Lambda_{x,h}$ has the distribution of a Brownian motion reflected on $\Gamma_0$ between $0$ and $x$ and coalescing with it outside this interval. As the point $(Y,\Gamma_0(Y))$ is $\Gamma_0$-measurable, conditionally on $\Gamma_0$, the distribution of the increments of $\Lambda_{Y,\Gamma_0(Y)}$ remains those of a Brownian motion starting at this point, reflected on $\Gamma_0$ between $0$ and $Y$ and coalescing with it outside this interval. It remains to use Proposition 2.2 (v) in \cite{TothWerner} which tells us that $\Lambda_{Y,\Gamma_0(Y)}$ is continuous to deduce that the distribution of this process corresponds indeed to the above description.
\end{proof}

\medbreak

A consequence of this lemma is that the distribution of $\sigma_x$ itself can be simply expressed in terms of areas under Brownian curves:
   \begin{align}\label{eqlawT+}
   \sigma_x \stackrel{(d)}{=} \sqrt{2} \left(\int_0^{x} |B_t| dt + \int_x^{\tau'} |B_t| dt\right)  
   \end{align}
   where  $B$ is a Brownian motion started at the origin and $\tau'$ denotes its first hitting time of $0$ after time $x$. 
Indeed, the initial curve $\Gamma_0(x - \cdot)$ has the distribution of a Brownian motion starting at $\Gamma_0(x)$ and the distribution of $\Gamma_x(x-\cdot)$ conditionally on 
$\Gamma_0$ is given by Lemma \ref{lem:BWalea}, thus the difference $\Gamma_{x}(x - \cdot) - \Gamma_0(x - \cdot)$ 
has the distribution of a reflected Brownian motion multiplied by $\sqrt{2}$, absorbed at its first hitting of $0$ after time $x$.

\subsection{Brownian estimates}\label{sec:ResultsBM}

As shown by the example of the law of $\sigma_x$, the construction of the TSRM via the Brownian web makes it possible to express the probability of TSRM-events in terms of Brownian motions and areas under Brownian curves. 
We now collect some results concerning the law of  Brownian motion integrals
that we will need later in the paper. 

Throughout this paper,  $B$ will denote a standard  Brownian motion, and  $\tilde{B}$ a reflected Brownian motion (that has the same law as $|B|$), $P_x$ will denote the law of these processes started at $x$. When $x=0$, we will sometimes simply write $P$ instead of $P_0$.
For each $y \in \R$, the first hitting time of the level $y$ by $B$ (respectively $\tilde{B}$) after time $t$ 
will be denoted by $\tau^{(t)}_y$ (resp. $\tilde{\tau}^{(t)}_y$), when $t=0$, we simply write $\tau_y$ (resp. $\tilde{\tau}_y)$.

In order to derive our estimates about the tail of $X_1$ and $H_1$, we will build on the following rather classical asymptotics about the areas under a Brownian motion and a Brownian bridge. The first two results can for instance be found in \cite{AreaJanson} and the very classical third one in \cite{Areatill0}.
Here and throughout the paper, 
$$\kappa := 2 |a'_1|^3 / 27$$
 where $a'_1$ denotes the first (negative) zero of the derivative of the Airy function Ai.

\begin{propo}\label{propo:resultsBM}
\begin{enumerate}
 \item For some positive constant $\gamma$, when $\eps \to 0$, 
$$P_0\left(\int_0^1 |B_t| dt \le \eps\right) \sim \gamma \, \eps \exp\left(-\frac{\kappa}{\eps^2}\right).
$$
 \item In the case of the Brownian bridge,
$$
P_0\left(\int_0^1 |B_t - t B_1| dt \le \eps\right) \sim \gamma' \exp\left(-\frac{\kappa}{\eps^2}\right)
$$ as $\eps \to 0$
for some positive constant $\gamma'$. 
\item
The law of the area under a Brownian motion starting at $1$ stopped at its first hitting of $0$ is given by 
$$
P_1\left(\int_0^{\tau_0} B_t dt \le u^{-3} \right) = \frac { \int_u^{\infty} e^{-2y^3/9} dy } { \int_0^\infty e^{-2y^3/9} dy } .
$$
\end{enumerate}
\end{propo}

This last statement follows in fact directly from the fact that the function 
$F(x,A) := P_x (\int_0^{\tau_0} B_t dt \le A )$ is a function of $x/A^{1/3}$ that satisfies the PDE 
$(\partial^2_x - 2x \partial_A ) F =0$ (because $F(B_{t\wedge\tau_0},A-\int_0^{t \wedge \tau_0} B_s ds)$ is a martingale).

Suppose that $U_1$ and $U_2$ are independent copies of the random variable $\int_0^{\tau_0} B_t dt$ in statement 3. A simple consequence of that estimate that will shall use at some point is that when $x \to 0$, 
\begin {equation}
  P_1 ( U_1 + U_2 \le x ) = \exp ( -  8/(9x) + O ( \log (1/x))).
\label {U1+U2}
\end {equation}
Indeed a lower bound of $P_1 ( U_1 + U_2 \le x )$ is simply given by  $(P_1 ( U_1\le x/2 ))^2$. For the upper bound, a possible proof consists in dividing the interval $[0,x]$ into $[1/x]+1$ intervals of length $x^2$ and to examine the probability that $U_1 + U_2  \le x$ according to which portion $U_1$ belongs to:
\begin{align*}
  P_1 ( U_1 + U_2 \le x ) \le \sum_{j=0}^{[1/x] + 1} P_1(U_1 \in [jx^2,(j+1)x^2]) \, P_1(U_2 \le x-jx^2+x^2)
\end{align*}
Using Proposition \ref{propo:resultsBM}-3, we deduce:
 \begin{align*}
  P_1 ( U_1 + U_2 \le x ) \le \sum_{j=0}^{[1/x] + 1} \exp\left(-\frac{2}{9\,x}\left[\frac{1}{(j+1)x} + \frac{1}{1+x(1-j)}\right] + O(\log(1/x))\right)
\end{align*}
The minimum over $j \in \{0,\cdots,[1/x] + 1\}$ of the function between the brackets takes the form $4(1+O(x))$. It gives the desired upper bound. \qed
\section{Tail estimates for the distribution of $X_1$}
\label {S.3}

\noindent The main goal of the present section is to derive the following fact:
\begin{propo}\label{propo:tailsX1}
When $x \to \infty$,
\begin{align*}
 \PP(X_1 >  x)  = \exp\left(- 2 {\kappa}\, x^3 + O(\ln(x))\right).
\end{align*}
\end{propo}

Note that $X_1$ and $-X_1$ have the same distribution, so that this also describes the behavior of $\PP(X_1 < -x)$ when $x  \to +\infty$
We would like to also point out that our proof can be easily adapted to the case when the initial condition is flat. The only difference is that the coefficient $2 {\kappa}$ in front of $x^3$ is replaced by $\kappa$ (because the corresponding Brownian motion is not multiplied by $\sqrt {2}$).

\begin{figure}
\centering
\includegraphics[width = 9cm]{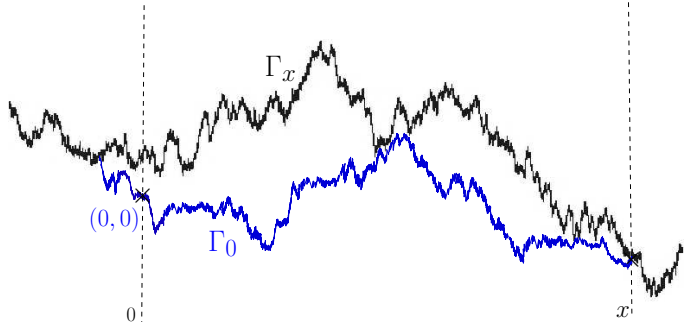}
\caption{The two reflected-coalescing curves $\Gamma_x$ and $\Gamma_0$}
\end{figure}

\begin{proof}
Recall the representation of the law of $\sigma_x$ from the end of Section \ref{subsec:TSRMandBW}.
It follows that
$$
 \PP\left(\sup_{s\in [0,1]} X_s \ge x \right) = 
\PP ( \sigma_x \le 1 ) \le P \left( \sqrt {2} \int_0^x \tilde B_u du \le 1 \right) = 
 P\left(\int_0^1 \tilde{B}_u du \le \frac{1}{\sqrt{2x^3} }\right).
$$
Combined with Proposition \ref{propo:resultsBM}-1, this proves immediately the upper bound.

For the lower bound, it is sufficient to estimate the probability of a well-chosen subset of the event $\{X_1 > x\}$, 
that can be easily described using  the Brownian web.
In order to ensure that $X_ 1 > x$, it would for instance suffice that $\sigma_{x+1/x^2} < 1$ and 
that $X$ stays to the right of $x$ during a time-interval of length $1$ after $\sigma_{x+1/x^2}$. We will use a slight variation of this 
idea:
 Let $\tilde \Gamma_x$ denote the line corresponding to the first time at which the local time 
$L_t (x)$ of $X$ at $x$ exceeds $1/x$.  Let $\hat \Gamma_{x+1/x^2}$ denote the line corresponding to the first time at which the local time at $x+1/x^2$ 
exceeds $1/(2x)$, and finally let $\Gamma_0'$ be the line corresponding to the first time at which the local time at $0$ reaches $1/x$. We will evaluate the probability 
that the following four events hold simultaneously (see Figure \ref{tailsX1quatrecourbes}. for a representation of those events):
\begin {itemize}
 \item The integral of $\tilde \Gamma_x  - \Gamma_0$ over $[0,x]$ does not exceed $1-2/x^3$ and $\tilde \Gamma_x(0) < 1/x^4$.
 \item The integral of $\Gamma_0' - \Gamma_0$ on $(- \infty, 0)$ does not exceed $1/x^3$.
 \item $\hat \Gamma_{x+1/x^2} (x) - \Gamma_0 (x) \le 1/x$ and $\hat \Gamma_{x+1/x^2} - \Gamma_0$ hits $0$ on $[x, x+1/x^2]$, and the integral of this function on 
$[x, x+1/x^2]$ does not exceed $1/x^3$.
 \item The integral of $\hat \Gamma_{x+1/x^2} - \Gamma_0$ on $[x+1/x^2, \infty)$ is greater than one.
\end {itemize}

\begin{figure}
\centering
\includegraphics[width = 14cm]{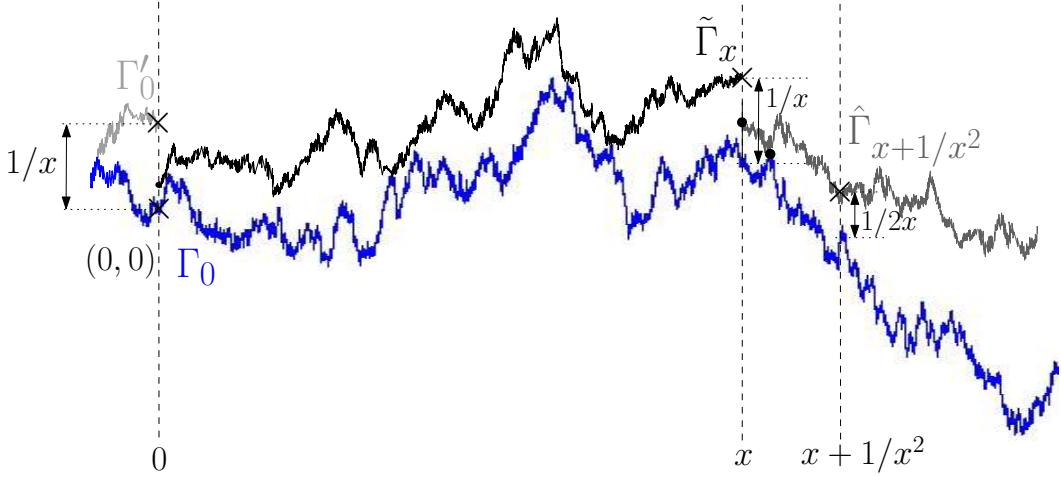}
\caption{The BW-curves $\Gamma_0$, $\Gamma'_0$, $\tilde \Gamma_x$ and $\hat \Gamma_{x+1/x^2}$}
\label{tailsX1quatrecourbes}
\end{figure}

It is easy to check just using monotonicity of the BW that if these four events hold then $X_1$ will be bigger than $x$ -- i.e. to the right of $x$ in the two-dimensional picture (the first, second and third one imply 
$\sup_{s \le 1} X_s \ge x+1/x^2$, the third and last one ensure that $X_1$ stays above $x$ during the time-interval $[\sigma_{x+1/x^2}, \sigma_{x+1/x^2}+1]$). Notice also that 
these four events are independent as the processes defining them (restricted to the appropriate time-intervals) correspond to different parts of the BW (and this is why we chose to work with  these events). Let us evaluate the probability of each of them. Thanks to  Brownian scaling, the second and the third one are equal to positive constants independent of $x$.

If the process $\hat \Gamma_{x+1/x^2} - \Gamma_0$ stays above $1/(4x)$ in the time interval of length $4x$ starting at $x +1/x^2$, then the fourth event is satisfied. It implies that the probability of the fourth event is bounded from below by 
\begin{align}
 P_0\left(\inf_{s \le 4x} B_s \ge -1/(4\sqrt{2}x)\right) = P_0\left(|B_1| \le 1/(8\sqrt{2}x^{3/2})\right) \ge c/x^{3/2} \label{fourthevent}
\end{align}
for some absolute constant $c$.

The probability of the first one is responsible for the main exponential term: The strong Markov property shows that it is bounded from below by
$$
P_{1/(\sqrt{2}x)}\Big(\tau_0 < 1/x^2,\, \int_0^{\tau_0} B_t dt \le 1/(\sqrt{2}x^3)\Big) \, P_0\left(\sqrt{2}\int_0^x |B_t| dt \le 1-3/x^3,\, \sqrt{2} |B_x| \le 1/x^4\right).
$$
The scaling property shows again that the first term in this product does not depend on $x$.  The second term can be evaluated thanks to the Brownian bridge. Scaling shows that it is bounded from below by:
\begin{align*}
 &P_0\left(\int_0^1 |B_t - tB_1| dt \le \frac{1}{\sqrt{2} x^{3/2}}\left(1-3/x^3\right) - \frac{|B_1|}{2},\, |B_1| \le \frac{1}{\sqrt{2} x^{9/2}}\right) \\
&\ge P_0\left(\int_0^1 |B_t - t B_1| dt \le \frac{1}{\sqrt{2} x^{3/2}}\left(1-4/x^3\right)\right) \times P_0\left( |B_1| \le \frac{1}{\sqrt{2} x^{9/2}}\right)
\end{align*}
because of the independence between $(B_t - t B_1,\,t \in [0,1])$ and $B_1$. 
Putting the pieces together, we get finally that 
\begin{align}
P ( X_1 \ge x ) \ge \frac{c'}{x^{6}} \times P_0\left(\int_0^1 |B_t - t B_1| dt \le \frac{1}{\sqrt{2} x^{3/2}}\left(1-4/x^3\right)\right) \label{lastprobaX1}
\end{align}
where $c'$ is some absolute constant. Proposition \ref{propo:resultsBM}-2 then allows to conclude.
\end{proof}

\section{Law of the iterated logarithm for $X$}
\label {S.4}

\subsection {Statement and proof of the upper bounds}
The main goal of this section is to use the previous estimates in order to derive the analogue for $X$ of the law of the iterated logarithm:

\begin {propo} \label{propo:LIL}
Almost surely
$$
\limsup_{t \to +\infty} \frac{X_t}{t^{2/3} \left(\ln\ln(t)\right)^{1/3}} =  \limsup_{t \to 0+} \frac{X_t}{t^{2/3} \left(\ln\ln(1/t)\right)^{1/3}}  = 1/(2\kappa)^{1/3}.
$$
\end {propo}

Stationarity shows that this also describes the almost sure fluctuations at any given positive time $t_0$ i.e. that almost surely,
$$
\limsup_{t \to 0+} \frac{X_{t_0+t} - X_{t_0}} {t^{2/3} \left(\ln\ln(1/t)\right)^{1/3}} = 1/(2\kappa)^{1/3}.
$$
The same type of local result will hold for the TSRM with flat initial condition at any given positive time. However, if $X$ is the TSRM with flat initial conditions, then the result stated in the proposition does not hold anymore. The proof can however be directly adapted and then shows that one just has to replace the constant $1/ (2 \kappa)^{ 1/3}$ by   $1/\kappa^{1/3}$.

\medskip
Let us now first briefly derive the upper bounds in this proposition i.e. the fact that these limsups are not greater than $1/ ( 2 \kappa)^{1/3}$. 
This part of the proof will go along similar lines as the standard proof of the LIL for the Brownian motion (see e.g., Chapter II p.~56 of \cite{RevuzYor}) based on Borel-Cantelli Lemmas. Let us first focus on the $t \to \infty$ part.
Clearly, it suffices to show that for some given $\lambda > 1$ and $\eps>0$, there almost surely exists some $N$ such that for all $n \ge N$, 
$$ \sup_{t \in [0,\lambda^n]}X_{t}
\le \frac{1+ \eps}{(2\kappa)^{1/3}} \, \lambda^{2n/3} \left(\ln\ln(\lambda^n)\right)^{1/3}.
$$
If we define 
$$x_n := \frac {1+\eps}{(2\kappa)^{1/3}} \,  (\ln\ln(\lambda^n))^{1/3}, $$
we get (because $\sup_{t \in [0,\lambda^n]}X_{t}/\lambda^{2n/3}$ and $\sup_{t \in [0,1]}X_t$ have the same law) from Proposition \ref {propo:tailsX1} that
$$
\PP\left(\lambda^{-2n/3} \,\sup_{t \in [0,\lambda^n]}X_{t} \ge x_n\right) = \PP\left(\sup_{t \in [0,1]}X_t \ge x_n\right) = e^{- 2 \kappa\, x_n^3 + O(\ln(x_n))}.
$$
Our choice for $x_n$ ensures that
$$\sum_{n} \PP\left(\lambda^{-2n/3} \,\sup_{t \in [0,\lambda^n]}X_{t} \ge x_n\right) <  \infty .$$
Note that $\eps$ can be chosen arbitrarily small which implies the result when $t \to \infty$.

The proof for $t \to 0$ is almost identical, except that we now have to choose $\lambda \in (0,1)$ and that the events we will consider are:
$$ \sup_{t \in [\lambda^n,\lambda^{n-1}]}X_{t}
\le \frac{1+ \eps}{(2\kappa)^{1/3}} \, \lambda^{2(n-1)/3} \left(\ln\ln(1/\lambda^{n-1})\right)^{1/3}.
$$
The result follows again using scaling.

\subsection {Proof of the lower bounds}

The purpose of this subsection is to derive the lower bounds in Proposition \ref {propo:LIL}.
 Let us stress that some caution is needed because the process $X$
 does not have independent increments, so that one has the standard proof of the LIL for Brownian motion can not be adapted directly.
 
We again first focus on the case where $t \to +\infty$.
Let us fix any small $\delta$. Our goal is to show that for $c := 1/(2\kappa)^{1/3}$ one can almost surely find a 
sequence of times $t_n \to +\infty$, such that
\begin{align}\label{infnboftimes}
X_{t_n} \ge (c - \delta) \, t_n^{2/3} \left(\ln\ln(t_n)\right)^{1/3}.
\end{align}

We will choose $t_n$ to be some first hitting times. More precisely, let us choose $\lambda > 1$, $\eps \in (0,2/3)$ and define for each 
$n \ge 1$, 
$$\lambda_n = \lambda^{n^{1+\eps}},$$ and let 
\begin{align*}
\tilde \sigma_n := \sigma_{\lambda_n} = \inf\{t \ge 0 \;:\; X_t = \lambda_n\}.
\end{align*}
Our sequence $(t_n)$ will be a subsequence of $(\tilde \sigma_n)$. 

Note that $\lambda_{n} / \lambda_{n-1} \sim \lambda^{(1+ \eps)n^{\eps}}$ increases quite rapidly when $n \to \infty$,
but not too fast either (both facts will be useful in our proof).
Define
 $$\gamma_n := c'\,\lambda_n^{3/2}\bigg{/}\sqrt{\ln\ln(\lambda_n^{3/2})}$$
where the positive constant $c'$ will be chosen later.

\begin{figure}
\centering
\includegraphics[width = 10cm]{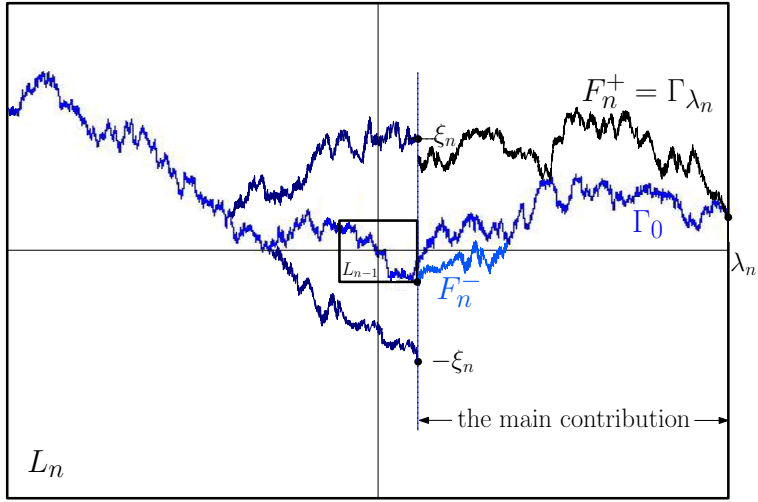}
\caption{Boxes $L_n$ and $L_{n-1}$ and the curves involved in $\mathcal{A}_n$ and $\mathcal{B}_n$}
\label{LIL}
\end{figure}

Our goal is to prove that $\tilde \sigma_n \le \gamma_n$ (i.e. that the area between $\Gamma_{\lambda_n}$ and  $\Gamma_0$ does not exceed  $\gamma_n$) infinitely often as soon as $c' > \sqrt{2 \kappa}$, which indeed implies (\ref{infnboftimes}).

Let us define the boxes $L_n := [-\lambda_n,\lambda_n] \times [-\eta_n,\eta_n]$ with $\eta_n := 3\sqrt{ \lambda_n \ln(n)} $. As the sequence $(\lambda_n)$ increases fast, the box $L_{n-1}$ is really small compared to $L_n$ when $n$ is large. 

Our choice for $\eta_n$ ensures that if we define
$$\mathcal{D}_n := \{\Gamma_0([-\lambda_{n},\lambda_n]) \in [-\eta_n,\eta_n] \}$$
then 
$$ \sum_n \PP ( \mathcal {D}_n^c  ) < \infty,$$
so that almost surely, $\mathcal {D}_n$ holds for all large enough $n$. 
Similarly, one can also for instance see that 
$$\Gamma_{\lambda_n} ([\lambda_{n-1},\lambda_n]) \in [-\eta_n,\eta_n]$$
almost surely for all large enough $n$.

The fact that the events $\{\tilde \sigma_n \le \gamma_n\}$ for $n \ge 1$ are not independent leads us to define closely related events that happen to be independent, so that we will be able to apply Borel-Cantelli arguments. The events that we are going to focus on will be defined in terms of the Brownian Web in the disjoint portions $(L_n \setminus L_{n-1})$. 
One minor technical difficulty is that in order to recognize where $\Gamma_0$ is in $L_n \setminus L_{n-1}$, one needs information about the Brownian web in $L_{n-1}$. We will circumvent this 
problem by considering instead the forward line in the web denoted by $F_{n}^-$ 
started from the bottom right corner of $L_{n-1}$. Then, we define $F_n^+$ to be the backward line in the web that is started from $F_n^- (\lambda_n)$ reflected above this curve $F_n^-$.

Now, we define the event ${\mathcal A}_n$ that the following three events hold:
\begin {itemize}
\item The area between $F_n^+$ and $F_n^-$ is small i.e. 
$$\int_{\lambda_{n-1}}^{\lambda_{n}} (F_n^+ (u) - F_n^- (u)) \,du \le (1-\eps) \gamma_n.$$
\item $F_n^+ ( \lambda_{n-1} ) \in [ \eta_{n-1}, \xi_n ]$ with $\xi_n := \eps c' \sqrt{\lambda_n/\ln(\ln(\lambda_n^{3/2})}$. 
\item $F_n^-$ and $F_n^+$ stay in $L_n$ during the interval $ [ \lambda_{n-1}, \lambda_n ]$.  
\end {itemize}
The last event ensures that ${\mathcal A}_n$ is indeed measurable with respect to the Brownian web in $L_n$. 
Note that, as before, the probability of this third event is very close to $1$ for $n$ large, and in fact equal to $1-a_n$ for some summable $a_n$.

We can use the same trick as in the proof of the lower bound of the tail of $X_1$ in order to get a lower bound for the probability that the first two events involved in this definition happen: Indeed, using scaling and then the independence between $t \in [0,1]\mapsto B_t - t B_1$ and $B_1$, we get that
\begin {align*}
 \PP(\mathcal{A}_n)  + a_n &\ge P\left(\int_0^1|B_u| \,du \le (1-\eps) c' \alpha_n,\; \sqrt{2} \eta_{n-1}/\sqrt{\lambda_n - \lambda_{n-1}}\le |B_1| \le \eps c' \alpha_n\right) \\
&\ge P\left(\int_0^1 |B_u - u B_1| \,du \le (1-\frac{3}{2}\eps) c' \alpha_n\right) \times P\left(|B_1| \in \Big[ \frac{\sqrt{2} \eta_{n-1}}{\sqrt{\lambda_n - \lambda_{n-1}}},  \eps c' \alpha_n\Big] \right)
\end{align*}
where $\alpha_n := 1/\sqrt{2 \,\ln\ln(\lambda_n^{3/2})}$. Part 2 of Proposition \ref{propo:resultsBM} then shows that 
$$\sum \PP(\mathcal{A}_n) = \infty $$ as soon as
 $c' \geq (1+\eps)^{1/2}/(1-3\eps/2) \times \sqrt{2 \kappa}$ (this is where we use that the sequence $(\lambda_n)$ is not increasing too fast).

Consider now the two backward lines started at $(\lambda_{n-1}, \xi_n )$ and $(\lambda_{n-1} , - \xi_n)$. 
Define the event $\mathcal {B}_n$ that the area between these two curves does not exceed $\xi_n^3$, that they coalesce in the interval
 $[\lambda_{n-1} - 2 \xi_n^2, \lambda_{n-1} - \xi_n^2]$,
 that they do not enter the box $[\lambda_{n-1}-\xi_n^2, \lambda_{n-1}] \times [-\xi_n/3, \xi_n/3]$ and do not exit the box $[\lambda_{n-1}-2\xi_n^2, \lambda_{n-1}] \times [-2\xi_n,2\xi_n]$.
Clearly, scaling shows that the probability of this event does not depend on $n$. Furthermore, our definition of $\xi_n$ ensures that for large enough $n$, 
one can check whether this event holds by just looking at the Brownian web in the part of $L_n \setminus L_{n-1}$ that is to the left of $\lambda_{n-1}$, which implies in particular that $\mathcal{B}_n$ is independent of $\mathcal{A}_n$.

Hence, it also follows that the events $(\mathcal{A}_n \cap \mathcal{B}_n)$ are independent, so that almost surely, $\mathcal{A}_n \cap \mathcal{B}_n$ holds for infinitely many values of $n$. As $\mathcal{D}_n$ holds almost surely for all large $n$, we conclude that almost surely
 $\mathcal{A}_n \cap \mathcal{B}_n \cap \mathcal{D}_{n-1}$ occurs infinitely often.
But we can notice that when this last event holds, then, due to the monotonicity properties of the Brownian web, we get that $F_n^- \leq \Gamma_0$, $F_n^-$ coalesces with $\Gamma_0$ in the interval $[\lambda_{n-1},\lambda_n]$ (because $F_n^+(\lambda_{n-1})$ is bigger than $\eta_{n-1}$) and thus $F_n^+ = \Gamma_{\lambda_n}$. Moreover, $\{F_n^+(\lambda_{n-1}) \leq \xi_n\} \cap \mathcal{D}_{n-1}$ implies that the backward lines involved in $\mathcal{B}_n$ enclose $\Gamma_0$ and $\Gamma_{\lambda_n}$. As $\xi_n^3$ is much smaller than $\eps \gamma_n$, it permits to conclude that $\mathcal{A}_n \cap \mathcal{B}_n \cap \mathcal{D}_{n-1}$ is included in $\tilde \sigma_n \leq \gamma_n$ as soon as $c'$ is greater than $(1+\eps)^{1/2}/(1-3\eps/2) \times \sqrt{2 \kappa}$. Taking the limit $\eps \to 0$ gives the result.

\bigskip
\noindent The proof of the lower bound when $t \to 0$ is almost identical. The very same proofs goes through without 
modification, one just has to take $\lambda$ smaller than $1$ instead of larger than $1$

\section{Fluctuations of the height}
\label {S.5}

\subsection {Statement of tail-estimates}

In this section, we will mostly study the tails of the distribution of the height $H_t$. 
Again, we can restrict ourselves to $t=1$ thanks to the scaling property. 
The estimates that we will derive are the following:
\begin{propo}
\label {tailprop}
There exists two positive constants $\eta$ and $\eta'$ such that for all large $h$,
\begin{align*}
\begin{array}{llcllll}
 \exp(-\eta h^{3/2})&\le &\PP(H_1 \le -h) &\le& \PP\left(\inf_{t \in [0,1]} H_t \le -h\right)&\le& \exp(-\frac{1}{\eta} h^{3/2}) \\
\exp(-\eta' h^{3/2})&\le &\PP(H_1 \ge h) &\le& \PP\left(\sup_{t \in [0,1]} H_t \ge h\right)&\le& \exp(-\frac{1}{\eta'} h^{3/2}).
\end{array}
\end{align*}
\end{propo}

We use here two different constants $\eta$ and $\eta'$ to stress that, unlike the case of $X$, 
 the distribution of $H$ is \textbf{not} symmetric (i.e.  the distributions of $-H_1$ and $H_1$ are quite different). See Fig. \ref{figureavechposhneg}.

\begin{figure}[!h]
\centering
\includegraphics[width = 12cm]{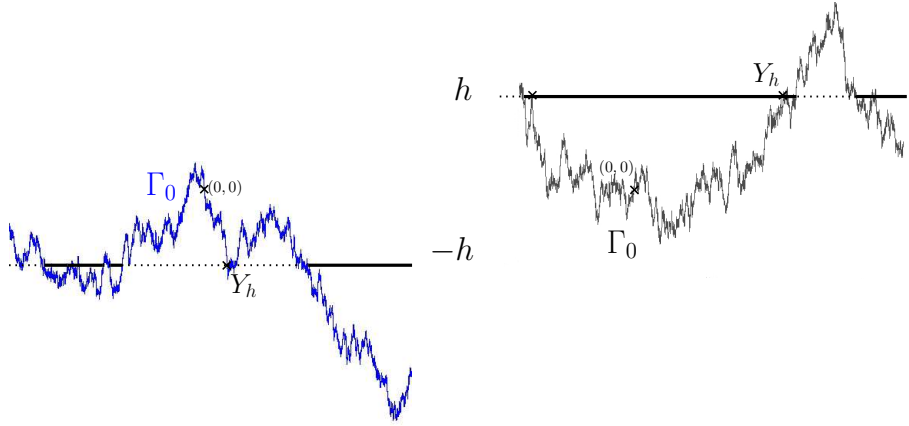}
\caption{The initial configuration $\Gamma_0$ and the lines $-h$ and $h$ (the thick lines represent the possible places $(X_1,H_1)$ where one can have $H_1 <-h$ (on the left hand side) and $H_1 >h$ (on the right hand side))}
\label{figureavechposhneg}
\end{figure}

In fact the derivation of the tail-estimates for $H_1$ are very different than those for $X_1$, because
 the initial profile will now play a key-role. Roughly speaking, the exceptional events that we will focus on will require a combination of a very favorable initial profile $\Gamma_0$
 and a particular behavior of the TSRM between time $0$ and $1$.

The next three subsections are devoted to the proof of Proposition \ref{tailprop}.

\subsection{Lower bounds}

We will first derive the lower bound for the probability that $H_1 \le -h$ and we will in fact focus on the sub-event $\{ H_1 \le -h \hbox { and }X_1 > 0 \}$. 
To guess what configuration to consider, we can imagine that for the initial profile, the random variable
$$Y_{-h} := \inf \{ y \ge 0 \ : \ \Gamma_0 (y) \le -h \} $$
 is exceptionally small. Then, on $[0, Y_{-h}]$, $\Gamma_0$ will at first glance look like a non-horizontal line with negative strong slope $-\alpha$ 
(to be determined), and one can compute the cost for another Brownian motion going backwards and reflected on this slope, starting at the point 
$(h/\alpha, -h)$ in order to create an area less than $1$. 
One has to find a compromise between the cost of creating this initial configuration (which is roughly $P(\tau_{-h} < h/\alpha)$) 
and the cost of creating this small area with this slope. A back-of-the envelope calculation shows that a slope $\alpha$ of the order of $\sqrt{h}$ 
is close to optimal. In other words, we will roughly ask the initial profile $\Gamma_0$ to go down to level $-h$ during the interval $[0, \sqrt h ]$ 
(recall that the ``natural'' Brownian scaling would give an interval of length $h^2$), and then the TSRM to run exceptionally fast down this slope.

More precisely, let us describe the events that we will require to hold (Figures \ref{firstthreeevents}. and \ref{zoomaroundsqrth+2}. can help to quickly see what is going on). Define 
\begin{align*}
 \eps_h := 1/ ( 5 \sqrt {h+2}),
\end{align*}
and the function $f_h(\cdot)$ to be the linear function defined on $[0, \sqrt {h+2}]$ such that 
$f_h(0)= 0$ and $f_h ( \sqrt{h+2}) = - h - 2$. Define $U_h$ to be the tube of vertical width $\eps_h$ around $f_h$, and $V_h$ to be the same tube, but shifted vertically by $2 \eps_h$ so that $V_h$ lies just above $U_h$. In other words,
\begin {eqnarray*}
 U_h & = & \{ (x,l) \ : \ x \in [0, \sqrt {h+2}] \hbox { and } | l  - f_h (x) | < \eps_h \} \\
 V_h & = & \{ (x,l) \ : \ x \in [0, \sqrt {h+2}] \hbox { and } | l  - f_h (x) - 2 \eps_h | < \eps_h \}.
\end {eqnarray*}
Then, we will require that
\begin {itemize}
 \item The initial profile $\Gamma_0$ stays within $U_h$ for all $x \in [0, \sqrt {h+2}]$.
 \item The backward line starting at $(\sqrt {h+2}, -h-2 + 2 \eps_h)$ stays in $V_h$ for all $x \in [0, \sqrt {h+2}]$. 
 \item The backward lines starting at $(0,0)$ and at $(0,1)$ coalesce in such a way that the area between them is less than $1/10$, i.e. $\int_{-\infty}^0 ( \Lambda_{(0,1)} (y) - \Lambda_{(0,0)} (y)) dy \le 1/10$. 
 \item The forward lines starting at $(\sqrt {h+2}, -h-3)$, $(\sqrt {h+2}, -h-1)$ and at $(\sqrt {h+2}, -h - 1/2)$ do coalesce before reaching the height $-h$, and the area between the last two curves is greater than $1$.
 \item The backward line starting at $(\sqrt {h+2}, - h - 1/2)$ and at $(\sqrt {h+2}, \Gamma_0 (\sqrt { h+2}))$  do coalesce before reaching the height $-h$, and in a horizontal time-span smaller than one.
\end {itemize}
\begin{figure}[!h]
\centering
\includegraphics[width = 9cm]{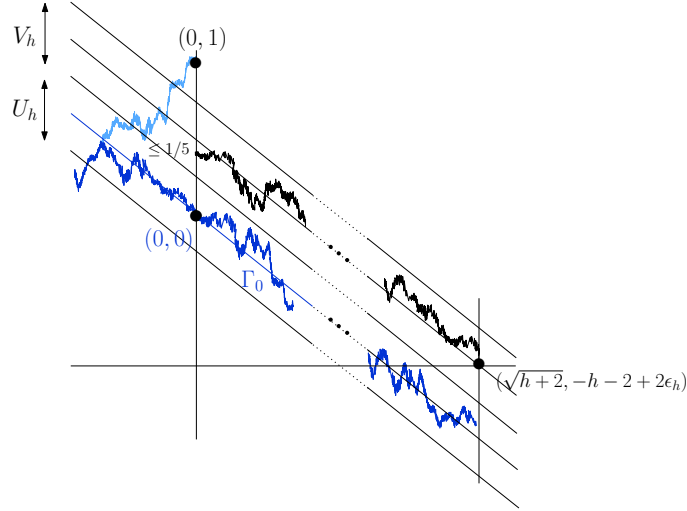}
\caption{Realization of the first three events}
\label{firstthreeevents}
\end{figure}
All these definitions may seem somewhat messy, 
but it is easy to check, just using the monotonicity properties of the Brownian web, 
that if all these events occur, then $H_t$ will hit $-h-2$ before time $1$, and that the process $(X_t, H_t)$ will stay under the horizontal 
line $-h$ for a time-interval of length at least one after this time. In particular, if the five events hold simultaneously, then $H_1 \le -h$.

\begin{figure}
\centering
\includegraphics[width = 7cm]{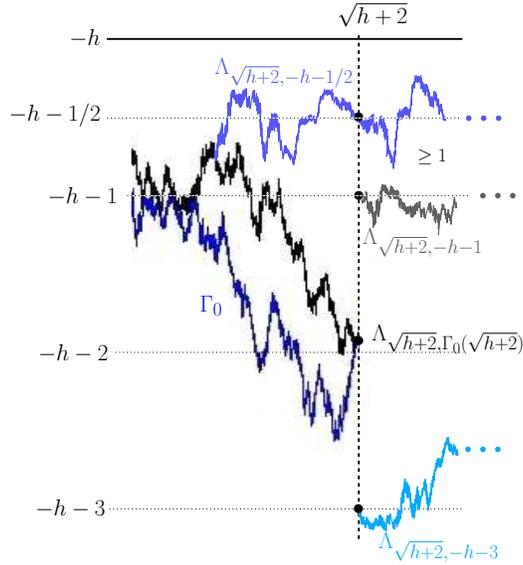}
\caption{Zoom around $t = \sqrt{h+2}$ with the lines involved for the realization of the fourth and fifth events}
\label{zoomaroundsqrth+2}
\end{figure}

The first four events are independent, because they correspond to events dealing with the Brownian web in disjoint domains. The conditional probability 
of the last one given the first four turns out to be bounded from below by a positive constant that does not depend on $h$. Indeed, it is independent on the 
third and fourth events. Moreover, the first and second events imply that the backward line started at $(\sqrt{h+2}, \Gamma_0(\sqrt{h+2}))$ stays in the tube 
$U_h \cup V_h$. 
Therefore the conditional probability is bounded from below by the probability that a standard Brownian motion hits the affine function $-f_h - 2$ before $1/2$ 
which clearly is positive (and bounded from below independently from $h$).

It remains to evaluate the probabilities of the first four events separately. The third and the fourth are positive and independent of $h$. 
The first two probabilities are equal. Note that $(\Gamma_0 (x), x \le \sqrt {h+2})$ is a Brownian motion and therefore 
$\Gamma^B (x):= \Gamma_0 (x) - x \Gamma_0 ( \sqrt {h+2})/\sqrt{h+2}$ is a Brownian bridge independent from $\Gamma_0(\sqrt {h+2})$. 
Furthermore, if 
$$\Gamma_0 (\sqrt {h+2}) \in  [-h-2-\eps_h /2, -h-2 + \eps_h /2 ]$$ and 
$$
\sup_{x \in [0, \sqrt {h+2}]} |\Gamma^B (x)|  \le \eps_h /2,  
$$
then the first event holds. 
The probabilities of each of these two independent events turns out of the type
 $\exp\big(- c h^{3/2}\big)$, which concludes the proof of the lower bound of $\PP (H_1 < -h)$.

\medbreak

The proof of the lower bound for $\PP( H_1 > h)$ is almost identical. The only difference is that the tubes now go upwards instead of downwards, 
and the reader can easily check that the same arguments work.

\subsection{Upper bound for $\PP ( \inf_{s \in [0,1]} H_s < - h)$}

We define again for $l>0$
$$Y_{-l} = \inf \{ y > 0 \ : \ \Gamma_0 (y) \le -l \}, $$
and we simply use $Y$ for $Y_{-h}$. Clearly, by symmetry,
$$ 
\PP \left( \inf_{s \in [0,1]} H_s < -h \right) = 2 \PP ( \sigma_{Y} < 1 ) \le 2 \PP \left( \int_0^{Y} ( \Gamma_{Y} (y) - \Gamma_0 (y)) dy  < 1 \right).
$$
Recall from Lemma \ref{lem:BWalea} that conditionally on $\Gamma_0$, the law of $\Gamma_{Y}$ is that of a backward Brownian motion started at $(Y, \Gamma_0 (Y))$ and reflected on $\Gamma_0$.

Note that Williams decomposition theorem (see for instance chapter $4$ Corollary $(4.6)$ p.~317 in \cite {RevuzYor}) states that the law of $(\Gamma_0 (Y-y) + h, y \in [0, Y])$ is that of a three-dimensional Bessel process up to its last passage time at level $h$.

Recall that by the strong Markov property for the Brownian motion, if one defines (for a given $\epsilon > 0$)
$$ \Gamma^j := ( \Gamma_0 (t+ Y_{-j \epsilon} ) + j \epsilon , t \in [ Y_{-j \epsilon}, Y_{-(j+1) \epsilon}] ) $$
then $\Gamma^0, \Gamma^1, \Gamma^2, \ldots$ are i.i.d. 
Let us choose $\epsilon = c / h^{1/2}$ for some large $c$, and denote by $N$ the integer part of $h/ \epsilon$. 

Monotonicity properties readily imply (by comparing $\Gamma_{Y}$ with the process where at each $Y_{- j \epsilon}$, the Brownian motion has to jump down to the actual location of $\Gamma_0$) that one can compare $\int_0^{Y_{-N \epsilon}} ( \Gamma_{Y} (y) - \Gamma_0 (y)) dy $ with the sum of $N$ i.i.d. copies of 
$\int_0^{Y_{-\epsilon}} ( \Gamma_{Y_{-\epsilon}} (y) - \Gamma_0 (y)) dy$ (the latter being stochastically dominated by the former). 
Hence, it finally suffices to evaluate the probability that the sum of $N$ copies of 
$$ \int_0^{Y_{-\epsilon}} ( \Gamma_{Y_{-\epsilon}} (y) - \Gamma_0 (y)) dy$$ 
is smaller than $1$. By scaling, this is exactly the same as the probability that the sum of $N$ copies of 
$$ \int_0^{Y_{-c}} ( \Gamma_{Y_{-c}} (y) - \Gamma_0 (y)) dy$$ 
is smaller than $h^{3/2}$ (which is smaller than $2cN$). 
Note that if we have chosen $c$ sufficiently large, we made sure (because of scaling) that 
$$ E \left( \int_0^{Y_{-c}} ( \Gamma_{Y_{-c}} (y) - \Gamma_0 (y)) dy \right) = c^{3} E \left( \int_0^{Y_{-1}} ( \Gamma_{Y_{-1}} (y) - \Gamma_0 (y)) dy \right) > 4c  $$
and it therefore follows from the standard 
 Cr\`amer Theorem for sums of i.i.d. positive random variables that for some positive constant $a$,
the probability in question is bounded from above by $ \exp ( -a  h^{3/2})$ for all large $h$, which concludes this part of the proof.
\qed

\subsection{Upper bound for $\PP ( \sup_{s \in [0,1]} H_s >  h)$}
\label{upper bound right tail}

Our goal is now to derive the upper bound for the probability that $H_t$ reaches a large positive $h$
before time $1$. In other words, we want to evaluate the probability that there exists an $x$ for which 
$S_{x,h} \le 1$. 
Note that for symmetry reasons, this probability is bounded from above by twice the probability that there exists a positive $x$ for which  $S_{x,h} \le 1$.

Note that the situation is different than in the previous section. Indeed, for $H_t$ to be negative before time $1$, the strategy had to be to find quickly a position where the initial profile was negative. Here, it could a priori happen that the $H_t$ is very large just because the TSRM spent some time in a tiny interval. So, the position at which this can happen is not a priori prescribed (see Figures \ref{diffconfigpourhpos} and \ref{fig:possiblepositions}).

\begin{figure}[!h]
\centering
\includegraphics[width = 10cm]{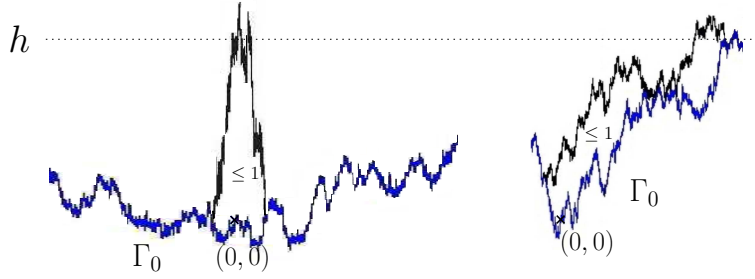}
\caption{Possible configurations for $H_1 > h$}
\label{diffconfigpourhpos}
\end{figure}

\begin{figure}[!h]
\centering
\includegraphics[width = 10cm]{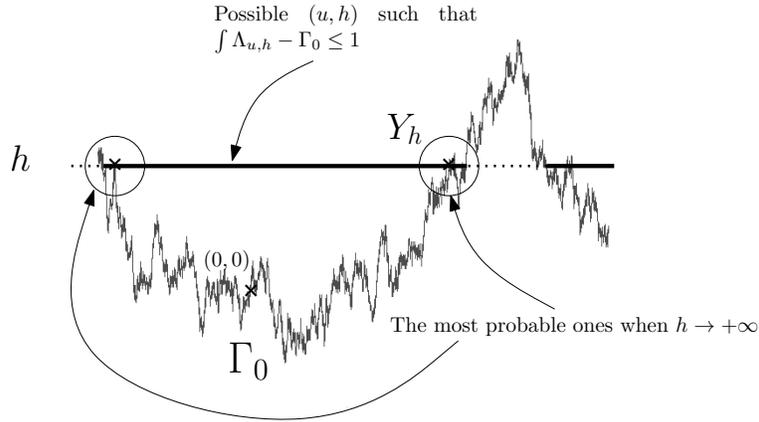}
\caption{The initial local time with the possible positions for $X$ when $H$ first hits the level $h$}
\label{fig:possiblepositions}
\end{figure}

Recall from our earlier estimates that the probability that the TSRM $X$ reaches $\sqrt {h}$ before time $1$ is bounded by $\exp ( -\kappa h^{3/2})$. It will therefore be sufficient to evaluate 
$$\PP ( \exists x \in [0, \sqrt {h} ] , \ S_{x,h} \le 1 ).$$

Also, it is easy to check that the probability that 
$ \Gamma [-\sqrt{h}, \sqrt{h} ] \notin [-h/4, h/4 ]$ 
is also very small, and bounded from above by $\exp ( - c h^{3/2})$ for some constant $c$ and all large $h$.

It remains to bound the probability that $ \Gamma [-\sqrt{h}, \sqrt{h} ] \in [-h/4, h/4 ]$ and
 $S_{x,h} \le 1$ for some $x \in [0, \sqrt {h}]$. In fact, we shall see that it is smaller than $\exp (-c h^3)$ for some constant $c$.  

Indeed, if this holds for some $x \in [0, \sqrt {h}]$, it means that the backward line in the BW starting from $(x,h)$ has to hit level $h/2$ in the interval $[x- 4/h, x]$. Indeed, otherwise, the domain in-between the initial profile and this backward line would contain a rectangle with area $(4/h) \times (h/4) =1$.

\begin{figure}[!h]
\centering
\includegraphics[width = 7cm]{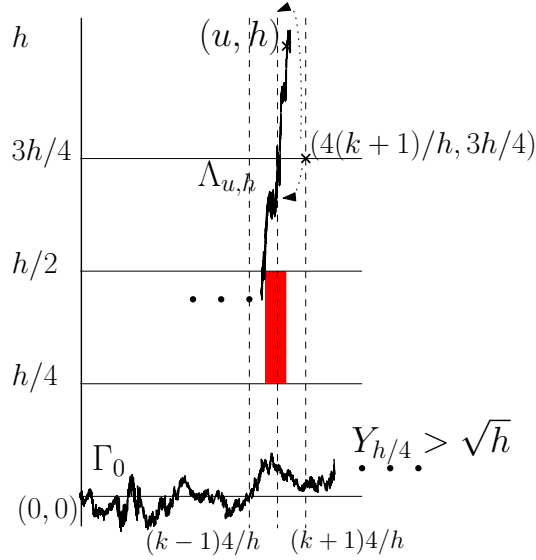}
\caption{Representation of $(u,h)$ verifying $S_{u,h} \le 1$}
\label{fig:repruh}
\end{figure}
Let us now suppose that for some $x \in [0, \sqrt {h}]$, the backward line in the BW starting from $(x,h)$ has to hit level $h/2$ in the interval $[x- 4/h, x]$.
Let us define $j$ to be the smallest integer such that $\tilde x := 4j/h \ge x$. Then, the backwards line starting from 
$(\tilde x, 3h/4)$ has to either hit $h$ or $h/2$ in the interval $[\tilde x - 2/h, \tilde x]$ (indeed, if it stays in the interval $[h/2,h]$, it would coalesce with the backward line starting from $(x,h)$ and therefore hit $h/2$). See Fig. \ref{fig:repruh}.

The probability that a Brownian motion started from the origin hits level $h/4$ before time $2/h$ decays very fast when $h \to \infty$ (a possible upper bound is of the type $\exp ( - c h^3)$). Note that there are of the order of $\sqrt {h} \times h / 2$ possibilities for $\tilde x$.

Putting the pieces together, we obtain an upper bound of the type
$$\PP ( \Gamma [-\sqrt{h}, \sqrt{h} ] \in [-h/4, h/4 ] \hbox { and } \exists x \in [0, \sqrt {h} ] , \ S_{x,h} \le 1 )
\le C' h^{3/2} e^{-ch^3}.$$
The upper bound for 
$\PP ( \sup_{s \in [0,1]} H_s >  h)$ follows.
\qed

\subsection{Flat initial condition}

It is worthwhile to note that for the flat initial conditions i.e. when $\Gamma_0=0$, the situation is completely different.
 Indeed, clearly $(H_t)$ is a non-negative process (so that there is no tail on the negative side...),
 and it is not possible to use a  ``favorable'' initial profile to help constructing an event where $H_1$ becomes very large.
In fact, the decay rate of the probability that $H_1$ is large is very different: 
\begin{propo}
When $h \to \infty$,
$\PP(H_1 \ge h) = \exp\left(-8h^3/9  + O(\ln(h))\right)$
\end{propo}

\begin{proof}
Let us start with the lower bound. Let us study the stopping times 
$S_{0, h_1}$ and $S_{0, h_2}$ for $h_1 = h + 1/h^2$ and $h_2= h + 5/h^2$ by $(X,H)$.
We would like to find an event that ensures that $S_{0, h_1} <1$ and that $H_t$ remains above $h$ during a time at least $1$ after this moment.
We will consider the following four events (here $\Lambda_1$ and $\Lambda_2$ denote the BW lines that go through $(0,h_1)$ and $(0,h_2$)):
\begin{itemize}
 \item $\Lambda_1 [-2/h^4,2/h^4] \subset [h+1/(2h^2),h+3/(2h^2)]$.
 \item $\Lambda_2 [-1/h^4,1/h^4] \subset [h+9/(2h^2),h+11/(2h^2)]$.
\item $\Lambda_1$ and $\Lambda_2$ coalesce in the vertical strip above  $[1/h^4, 2/h^4]$ 
and in the vertical strip above $[-2/h^4, -1/h^4]$
 (combined with the previous conditions, this implies that the area between $\Lambda_1$ and $\Lambda_2$ is greater than $6/h^6$ and that $H_t \ge h$ during the corresponding time-interval $[S_{0,h_1},S_{0,h_2}]$).
 \item The integral of $\Lambda_1$ on the interval $[2/h^4,+\infty)$ and on the interval $(-\infty,-2/h^4]$ both belong to $[1/2-2/h^3-3/h^6,1/2-2/h^3-2/h^6]$.
\end{itemize}
 
It is easy to see that $H_1 \ge h$ if those events hold simultaneously. Scaling shows that the probability that the first three are satisfied simultaneously is a constant that does not depend on $h$.
Using the simple Markov property, conditionally on the first events, $(\Lambda_1(2/h^2 + u), \, u \ge 0)$ is a Brownian motion starting at some level in $[1/(2h^2),3/(2h^2)]$. With the expression of the density of the area under a Brownian motion until its first passage time at $0$ given in Proposition \ref{propo:resultsBM}-3, we have:
\begin{align*}
 P_{h+u/h^2}\left(\int_0^{\tau_0} B_t dt \in [1/2-2/h^3-3/h^6,1/2-2/h^3-2/h^6]\right) \ge \exp\left(-4 h^3/9 + O(\ln(h))\right) 
\end{align*}
which is valid uniformly for every $u \in [1/2,3/2]$.

\noindent Therefore,
$$
\PP(H_1 \ge h) 
\ge \left(\exp\left(-4h^3/9 + O(\ln(h))\right)\right)^2 
= \exp\left(-8 h^3/9 + O(\ln(h))\right)
$$

For the upper bound, we can adapt the proof of the corresponding bound in the stationary case. The situation is at first sight simpler here, because we do not have to worry about the initial line. 

Let us denote the Brownian web with flat initial data by $(\Lambda'_{x,h}, (x,h) \in \R\times \R_+^*)$  and $S'_{x,h}$ the integral of $\Lambda'_{x,h}(\cdot)$ over $\R$. First, notice that symmetry and the tail estimates for $X_1$ show that it is sufficient to find an upper bound for
\begin{align*}
  \PP(\exists y \in [0, Ch] \,:\, S'_{y,h} \le 1)
\end{align*}
for some given large enough $C$.

We now divide the interval $[0, Ch]$ into circa $Ch^{10}$ smaller intervals $I_k := [x_k, x_{k+1}] = [k/h^{9},(k+1)/h^{9}]$ and we wish to bound  $\PP(y \in I_k,\, S'_{y,h} \le 1)$ for each $k$. 

We are going to consider two cases depending on whether $\Lambda'_{y,h} ( I_k) \subset  [ h-1/h^2, \infty)$ or not:

\begin{itemize}
 \item If $\Lambda'_{y,h} ( I_k) \subset  [ h-1/h^2, \infty)$, then $\Lambda'_{x_{k+1}, h-1/h^2}$ is below $\Lambda'_{y,h}$, so that $S'_{x_{k+1}, h-1/h^2} \le 1$. 
 \item If $\Lambda'_{y,h} ( I_k) \not\subset  [ h-1/h^2, \infty)$, then either the backward line started at $(x_{k+1}, h - 1 / (2h^2)$ or the forward line started at $(x_k, h -1/(2h^2))$ does not stay in $[h - 1/h^2, h )$ during the interval $I_k$. 
\end{itemize}

The probability of the second case is very small, and can be bounded by a constant times $ \exp ( -Ch^6)$. 
The probability of the first case is bounded by the probability that the area under a two-sided Brownian motion starting at the level $h-1/h^2$ until the first hitting times of $0$ (on both sides) is less than $1$. One can then conclude 
 using the estimate (\ref {U1+U2}), and summing over the $Ch^{10}$ values of $k$ (that correspond to another $e^{O(\log (h))}$ term).
\end{proof}

\subsection {Almost sure fluctuations}\label{sec:loi01}

Our tail-estimates for $H$ are less precise than those we obtained for $X$. However, let us say a few words on how to nevertheless 
deduce information about the almost sure behavior of the process $(H_t, t \ge 0)$:
\begin{corollary}\label{corol:LILH}
There exists four constants $l^+ > 0$, $l^+ _0> 0$, $l^- < 0$ and $l^-_0 < 0$ such that
\begin{align*}
\begin{array}{l}
\limsup_{t \to \infty} t^{-1/3} (\ln\ln(t))^{-2/3} \, H_t = l^+ \\
\liminf_{t \to \infty} t^{-1/3} (\ln\ln(t))^{-2/3} \, H_t = l^- \\
\limsup_{t \to 0} t^{-1/3} (\ln\ln(1/t))^{-2/3} \, H_t = l^+_0 \\
\liminf_{t \to 0} t^{-1/3} (\ln\ln(1/t))^{-2/3} \, H_t = l^-_0.
\end{array}
\end{align*}
\end{corollary}
Carefully adapting the proofs that we presented for $X$, using our tail estimates for the process $H$ yields statements of the following type: There exists two positive finite constants $\tilde l$ and $\hat l$ such that almost surely 
$$ \hat l \le \limsup_{t \to \infty} t^{-1/3} (\ln\ln(t))^{-2/3} \, H_t  \le \tilde l . $$
The upper bound is a direct consequence of the Borel-Cantelli Lemma, and the lower bound is obtained as in the case of $X$ by considering event that are measurable with respect to the information provided by Brownian web restricted to disjoint domains.
Let us briefly give the outline of the proof. As it is very similar to the fluctuations of $X$, we omit the details and just outline the proof:

Let us choose a sequence $(\lambda_n)$ increasing fast, but not too fast either ($\lambda_n := \lambda^{n(n-1)}$ with $\lambda > 1$ is suitable). It suffices to prove that there exists some absolute constant $c >0$ such that almost surely, the process $H$ reaches the height $\lambda_n$ before time $c \,\lambda_n^3/(\ln(\lambda_n))^2$ for infinitely many values of $n$. For each given $n$, we will focus on the first time at which the TSRM $X$ reaches the position $Y_{\lambda_n} := \inf\{y\ge 0\;:\; \Gamma_0(y) = \lambda_n\}$. Clearly, at that random time $\sigma_{Y_{\lambda_n}}$, the height $H$ is equal to $\lambda_n$.

Set $l_n := \lambda_n^2/\ln(n)$ and consider the boxes $L'_n := [-l_n, l_n] \times [-2\lambda_n, 2\lambda_n]$. It is easy to see via the  Borel-Cantelli Lemma that almost surely, for all but finitely many $n$, 
the event $$\mathcal{D}'_n := \{\Gamma_0 [-l_n, l_n, ]  \subset [ -2 \lambda_n, \lambda_n ]  \}$$
does hold. 

We introduce also $\xi_n := \alpha \, \lambda_n/\ln(n)$. As for the proof of the lower bound, we define two parallel upwards-going  tubes  $U_n$ and $V_n$ such that the bottom line of $U_n$ is the segment joining the points $(l_{n-1},-\xi_n + 2\lambda_{n-1})$ and $(l_n,\lambda_n)$ and the vertical width of $U_n$ is $\xi_n$. The tube $V_n$ is simply the same tube as $U_n$ but translated vertically by $\xi_n$.  
We consider the following three BW-curves: $F^-_n$ the BW-curve starting from $(l_{n-1},-\xi_n/2)$, $F^+_n$ the backward BW-curve starting from $(l_n,\lambda_n + 2 \xi_n)$ and $G^+_n$ the backward BW-curve starting at $(l_{n-1},2 \xi_n)$.
We will now study $\mathcal{A}'_n$ that the following events occurs:
\begin{itemize}
 \item $F^-_n$ stays in the tube $U_n$ and $F^+_n$ stays in the tube $V_n$.
 \item The integral of $G^+_n - F_n^-$ over $(-\infty,l_{n-1}]$ is less than $\xi_n^3$ and $G^+_n$ and $F_n^-$ do not enter in $L'_{n-1}$.
\end{itemize}
Notice that the events $\mathcal{A}'_n$ depend only on BW-curves in $L'_n \backslash L'_{n-1}$ and are therefore independent. Note also that the first event in $\mathcal{A}'_n$ is independent of the second one. The probability of the second event is bounded below by a constant. A similar computation to the proof of the lower bound permits to deal with the first one and shows that the series $\sum \PP(\mathcal{A}'_n)$ diverges. Thus almost surely $\mathcal{A}'_n$ holds infinitely often. Moreover, the values of the sequences $\xi_n$ and $l_n$ and BW monotonicity imply that $\mathcal{A}'_n \cap \mathcal{D}'_{n-1}$ is a sub-event of $\sigma_{Y_{\lambda_n}} \leq c \,\lambda_n^3/(\ln(\lambda_n))^2$ for some $c >0$ which does not depend on $n$. It proves the desired bound.

\bigskip

Let us now describe how to use a $0-1$ type argument in order to conclude.
We want for instance to show that
$$ Z := \limsup_{t \to \infty} \frac {H_t}{t^{1/3} (\ln \ln t)^{2/3}}
$$
is 
almost surely constant (the previous estimates then show that this constant is positive and finite).

Consider for any positive $h$, the curve $\Lambda_{(0,h)}$ started at height $h$ on the vertical axis. It is the profile at the stopping time corresponding to the first time at which $L_t (0)$ reaches $h$. Let us denote this random time by $\rho_h$. We know that $\rho_h \to \infty$ almost surely as $h \to \infty$. 

For all $h>0$, we denote by ${\mathcal G}_h$ the $\sigma$-field that contains all the information about the Brownian web above the line $\Lambda_{(0,h)}$. In other words, it is the $\sigma$-field generated by this line and by $((X_{t+ \rho_h}, H_{t+ \rho_h}), t \ge 0)$. Note that 
$Z$ is therefore ${\mathcal G}_h$ measurable (for all $h >0$). As ${\mathcal G}_h$ is decreasing with $h$, it follows that $Z$ is measurable with respect to ${\mathcal G}_\infty := \cap_h {\mathcal G}_h$.

For all positive $N$, let us now denote ${\mathcal V}_N$ the $\sigma$-field generated by the process $(X,H)$ up to the first time at which $\max (|X|, |H|)$ reaches $N$. 
Clearly, this stopping time is almost surely finite and when $N \to \infty$, it converges almost surely to $\infty$ because $(X,H)$ is a continuous process. 
Furthermore, any event in ${\mathcal V}_N$ can be read off by looking at the Brownian web lines inside the square $A_N= [-N,N]^2$. 

Suppose that $N$ is fixed, that $U$ is a $\sigma(Z)$-measurable event, and that $V$ is ${\mathcal V}_N$ measurable. Suppose furthermore that $W_{h,N}$ is the event that 
the line $\Lambda_{(0,h)}$ does not intersect $[-N, N]^2$. Clearly, the events $W_{h,N} \cap U$ and $V$ are independent as the former can be read off by looking only at the Brownian web outside of $[-N,N]^2$. On the other hand, we know that $P (W_{h,N}) \to 1$ as $h \to \infty$. Hence, 
$$P ( U \cap V ) = \lim_{h \to \infty} P ( U \cap V \cap W_{h,N} )  = P (V) \lim_{h \to \infty} P (U \cap W_{h,N}) =  P (U) P(V). $$
It follows that $U$ is independent of the $\sigma$-field generated by $\cup_N {\mathcal V}_N$, that contains $\sigma ( H_t, t \ge 0)$ and therefore also $U$. 
Hence, $P( U) = 0$ or $P (U)=1$.

\medbreak

The proof  of the fact that 
$$
Z' := \limsup_{t \to 0} \frac {H_t }{ t^{1/3} (\ln \ln (1/t))^{2/3}}$$ is almost surely constant is similar. We know that almost surely $\rho_h \to 0$ as $h \to 0$, and that $H$ is continuous. It follows that the process $(H_t, t \ge 0)$ is measurable with respect to $\sigma ( \cup_h {\mathcal G}_h)$. But for any fixed $h_0 > 0$, the probability that $\Lambda_{(0,h_0)}$ intersects the box $[-1/N, 1/N]^2$ 
goes to $0$ as $N \to \infty$, and on the other hand, we know that $Z'$ is measurable with respect to each ${\mathcal V}_{1/N}$ (because $\rho_{1/N} > 0$). Hence, it follows readily that $Z'$ is independent of ${\mathcal G}_{h_0}$, and then, letting $h_0 \to 0$ that the random variable $Z'$ is independent of itself and therefore constant. 
\bibliographystyle{plain}
\bibliography{bibliotsrm}

D\'epartement de Math\'ematiques et Applications

Ecole Normale Sup\'erieure 

75230 Paris cedex 05

France 

\medbreak

laure.dumaz@ens.fr

\end{document}